\documentclass[10pt,leqno]{amsart}%
\usepackage{amsmath}
\usepackage{amsfonts}
\usepackage{amssymb}
\usepackage{graphicx}%
\providecommand{\U}[1]{\protect\rule{.1in}{.1in}}
\newtheorem{Thm}{Theorem}[section]
\newtheorem{Lem}[Thm]{Lemma}

\newtheorem{corollary}[Thm]{Corollary}
\theoremstyle{definition}

\newtheorem{example}[Thm]{Example}

\numberwithin{equation}{section}

\newcommand{\cA}{{\mathcal A}}

\newcommand{\field}[1]{\mathbb{#1}}

\newcommand{\R}{\field{R}}

\newcommand{\C}{\field{C}}
\newcommand{\hyp}{\field{H}}

\DeclareMathOperator\diver{div}
\DeclareMathOperator\Hess{Hess}

\DeclareMathOperator\spt{spt}

\DeclareMathOperator\sect{Sect}
\DeclareMathOperator\ric{Ric}
\DeclareMathOperator\Int{Int}
\begin{document}
\title[Asymptotic Dirichlet problem]{On the asymptotic Dirichlet problem for the minimal hypersurface equation in a
Hadamard manifold}
\author[Jean-Baptiste Casteras]{Jean-Baptiste Casteras}
\address{UFRGS, Instituto de Matem\'atica, Av. Bento Goncalves 9500, 91540-000 Porto
Alegre-RS, Brasil.}
\email{Jean-Baptiste.Casteras@univ-brest.fr}
\author[Ilkka Holopainen]{Ilkka Holopainen}
\address{Department of Mathematics and Statistics, P.O. Box 68, 00014 University of
Helsinki, Finland.}
\email{ilkka.holopainen@helsinki.fi}
\author[Jaime B. Ripoll]{Jaime B. Ripoll}
\address{UFRGS, Instituto de Matem\'atica, Av. Bento Goncalves 9500, 91540-000 Porto
Alegre-RS, Brasil.}
\email{jaime.ripoll@ufrgs.br}
\thanks{J.-B.C. supported by the CNPq (Brazil) project 501559/2012-4; I.H. supported by the Academy of Finland, project 252293; J.R. supported by the CNPq (Brazil) project 302955/2011-9}
\subjclass[2000]{58J32, 53C21, 31C45}
\keywords{minimal graph equation, Dirichlet problem, Hadamard manifold}

\begin{abstract}
We study the Dirichlet problem at infinity on a Cartan-Hadamard manifold
$\left(  M,d\right)  $ of dimension $n\geq2$ for a large class of operators
containing in particular the $p$-Laplacian and the minimal graph operator. We
extend the existence result of \cite{HoVa} obtained for the $p$-Laplacian to
our class of operators. As an application of our main result, we prove the
solvability of the asymptotic Dirichlet problem for the minimal graph equation
for any continuous boundary data provided that
\[
-d(o,x)^{2\left(  \phi-2\right)  -\varepsilon}\leq\sect_{x}(P)\leq-\dfrac
{\phi(\phi-1)}{d(o,x)^{2}},
\]
for some constants $\phi>1$ and $\varepsilon>0,$ where $o$ is a fixed point of
$M$, $\sect_{x}(P)$ is the sectional curvature of a plane $P\subset T_{x}M$
and $x$ is any point in the complement of a ball $B(o,R_{0})$. So far, the
solvability of the asymptotic Dirichlet problem in the minimal case was
established only under hypothesis which included the condition $\sect_{x}%
(P)\leq c<0$ (see \cite{GR}, \cite{RT}).
\end{abstract}
\maketitle










\section{Introduction}

\label{SecIntro}
In this paper we study the asymptotic Dirichlet problem for operators
\begin{equation}
\label{M_equ}{\mathcal{Q}}[u]:= \diver{\mathcal A}(\lvert\nabla u \rvert
^{2})\nabla u
\end{equation}
on Cartan-Hadamard manifolds with ${\mathcal{A}}$ subject to growth
conditions. Recall that a Cartan-Hadamard manifold is a complete, connected
and simply connected Riemannian $n$-manifold, $n\ge2$, of non-positive
sectional curvature. By the Cartan-Hadamard theorem, the exponential map
$\exp_{o}\colon T_{o}M\to M$ is a diffeomorphism for every point $o\in M$.
Consequently, $M$ is diffeomorphic to $\mathbb{R}^{n}$. A Cartan-Hadamard
manifold $M$ can be compactified by adding a \emph{sphere at infinity\/},
denoted by $M(\infty)$, so that the resulting space $\bar M=M\cup M(\infty)$
equipped with the so-called \emph{cone topology\/} is homeomorphic to
a closed Euclidean ball; see \cite{EO}. The \emph{asymptotic Dirichlet
problem\/} on $M$ for the operator ${\mathcal{Q}}$ is then the following:
Given a continuous function $h$ on $M(\infty)$ does there exist a (unique)
function $u\in C(\bar M)$ such that ${\mathcal{Q}}[u]=0$ on $M$ and $u\vert
M(\infty)=h$?

We assume that ${\mathcal{A}}\colon(0,\infty)\to[0,\infty)$ is a smooth
function such that
\begin{equation}
\label{Agrowth}{\mathcal{A}}(t)\le A_{0}t^{(p-2)/2}%
\end{equation}
for all $t>0$, with some constants $A_{0}>0$ and $p\ge1$, and that
${\mathcal{B}}:={\mathcal{A}}^{\prime}/{\mathcal{A}}$ satisfies
\begin{equation}
\label{Bgrowth}-\frac{1}{2t} < {\mathcal{B}}(t)\le\frac{B_{0}}{t}%
\end{equation}
for all $t>0$ with some constant $B_{0}>-1/2$. Furthermore, we assume that
$t{\mathcal{A}}(t^{2})\to0$ as $t\to0+$ and therefore we set ${\mathcal{A}%
}(\lvert X \rvert^{2})X=0$ whenever $X$ is a zero vector. As a consequence of
\eqref{Bgrowth}, the function $t\mapsto t{\mathcal{A}}(t^{2})$ is strictly
increasing. A function $u$ is a (weak) solution to the equation ${\mathcal{Q}%
}[u]=0$ in an open set $\Omega\subset M$ if it belongs to the local Sobolev
space $W^{1,p}_{\mathrm{loc}}(\Omega)$ and
\begin{equation}
\label{Qsol}\int_{\Omega}\bigl\langle {\mathcal{A}}\bigl(\lvert\nabla u
\rvert^{2}\bigr)\nabla u,\nabla\varphi\bigr\rangle dm =0
\end{equation}
for every $\varphi\in C^{\infty}_{0}(\Omega)$. Such function $u$ will be
called a \emph{${\mathcal{Q}}$-solution\/} in $\Omega$. Furthermore, we say
that a function $u\in W^{1,p}_{\mathrm{loc}}(\Omega)$ is a \emph{${\mathcal{Q}%
}$-subsolution\/} in $\Omega$ if ${\mathcal{Q}}[u]\ge0$ weakly in $\Omega$,
that is
\begin{equation}
\label{Qsub}\int_{\Omega}\bigl\langle{\mathcal{A}}\bigl(\lvert\nabla u
\rvert^{2}\bigr)\nabla u, \nabla\varphi\bigr\rangle\,dm\le0
\end{equation}
for every non-negative $\varphi\in C^{\infty}_{0}(U)$. Similarly, a function
$v\in W^{1,p}_{\mathrm{loc}}(U)$ is called a \emph{${\mathcal{Q}}%
$-supersolution\/} in $\Omega$ if $-v$ is a ${\mathcal{Q}}$-subsolution in
$\Omega$. Note that $u+c$ is a ${\mathcal{Q}}$-solution (respectively,
${\mathcal{Q}}$-subsolution, ${\mathcal{Q}}$-supersolution) for every constant
$c$ if $u$ is a ${\mathcal{Q}}$-solution (respectively, ${\mathcal{Q}}%
$-subsolution, ${\mathcal{Q}}$-supersolution). It follows from the growth
condition \eqref{Agrowth} that test functions $\varphi$ in \eqref{Qsol} and
\eqref{Qsub} can be taken from the class $W^{1,p}_{0}(\Omega)$ if
$\lvert\nabla u \rvert\in L^{p}(\Omega)$.

We call a relatively compact open set $\Omega\Subset M$ \emph{${\mathcal{Q}}%
$-regular\/} if for any continuous boundary data $h\in C(\partial\Omega)$ there
exists a unique $u\in C(\bar{\Omega})$ which is a ${\mathcal{Q}}$-solution in
$\Omega$ and $u\vert\partial\Omega=h$. In addition to the growth conditions on
${\mathcal{A}}$, we assume that

\begin{itemize}
\item[(A)] there is an exhaustion of $M$ by an increasing sequence of
${\mathcal{Q}}$-regular domains $\Omega_{k}$, and that

\item[(B)] locally uniformly bounded sequences of continuous ${\mathcal{Q}}%
$-solutions are compact in relatively compact subsets of $M$.
\end{itemize}

In this paper the primary example of the equations that satisfy the conditions
above is the \emph{minimal graph equation\/}
\begin{equation}
{\mathcal{M}}[u]:=\diver\dfrac{\nabla u}{\sqrt{1+\lvert\nabla u\rvert^{2}}%
}=0,\label{min_eqn}%
\end{equation}
in which case
\[
{\mathcal{A}}(t)=\frac{1}{\sqrt{1+t}}\text{ and }{\mathcal{B}}(t)=-\frac
{1}{2(1+t)},
\]
and therefore \eqref{Agrowth} and \eqref{Bgrowth} hold with constants
$A_{0}=1$ and $B_{0}=0$, respectively. We note that $u$
satisfies \eqref{min_eqn} if and only if 
$G:=\left\{  \left(  x,u(x)\right)\colon x\in\Omega\right\}  $ 
is a minimal hypersurface in $M\times\mathbb{R}$. For the minimal graph equation, condition (A) follows
from \cite[Theorem 2]{DLR} where $\Omega_{k}$ may be chosen as a
geodesic ball with radius $k$ centered at a fixed point of $M$, and condition (B) follows 
from \cite[Theorem 1.1]{Spruck} (see also \cite[Theorem 1]{DLR}).

The class of equations considered here includes also the $p$-Laplace equation
\[
\diver\bigl(\lvert\nabla u\rvert^{p-2}\nabla u\bigr)=0,\ 1<p<\infty,
\]
in which case
\[
{\mathcal{A}}(t)=t^{(p-2)/2}\text{ and }{\mathcal{B}}(t)=\frac{p-2}{2t},
\]
and so $A_{0}=1$ and $B_{0}=(p-2)/2$. In the special case $p=2$ one obtains
the usual Laplace-Beltrami equation $\Delta u=0$, with 
${\mathcal{A}}(t)\equiv1$ and ${\mathcal{B}}(t)\equiv0$. It is well-known that the
properties (A) and (B) above hold for the $p$-Laplace equation and that (weak)
solutions of the $p$-Laplace equation have H\"{o}lder-continuous
representatives, usually called $p$-harmonic functions; see \cite{HKM}.

The asymptotic Dirichlet problem for the Laplace-Beltrami operator was solved
affirmatively by Choi \cite{choi} under assumptions that sectional curvatures
satisfy $\sect\leq-a^{2}<0$ and any two points in $M(\infty)$ can be separated
by convex neighborhoods. Such appropriate convex sets were constructed by
Anderson \cite{andJDG} for manifolds of pinched sectional curvature
$-b^{2}\leq\sect\leq-a^{2}<0.$ Independently, Sullivan \cite{sullivan} solved
the Dirichlet problem at infinity under the same pinched curvature assumption
by using probabilistic arguments. In \cite{andschoen}, Anderson and Schoen
presented a simple and direct solution to the Dirichlet problem again in the
case of pinched negative curvature. By modifying Anderson's argument, Borb\'{e}ly
\cite{borbpams} was able to construct appropriate convex sets under a weaker curvature 
lower bound $\sect_x\ge -g\bigl(\rho(x)\bigr)$, where $g(t)\approx e^{\lambda t}$, 
with $\lambda<1/3$. Here and throughout the paper $\rho(x)$ stands for the distance 
between $x\in M$ and a fixed point $o\in M$.
Major contributions to the Dirichlet
problem were given by Ancona in a series of papers \cite{ancannals},
\cite{anchyp}, \cite{ancpot}, and \cite{ancrevista}. In particular, he was
able to replace the curvature lower bound with a bounded geometry assumption
that each ball up to a fixed radius is $L$-bi-Lipschitz equivalent to an open
set in $\mathbb{R}^{n}$ for some fixed $L\geq1$; see \cite{ancannals}. On the
other hand, in \cite{ancrevista} Ancona constructed a 3-dimensional
Cartan-Hadamard manifold with sectional curvatures bounded from above by $-1$
where the asymptotic Dirichlet problem is not solvable. Another example of a
(3-dimensional) Cartan-Hadamard manifold, with sectional curvatures $\leq-1$,
on which the asymptotic Dirichlet problem is not solvable was constructed by
Borb\'{e}ly \cite{Bor}. To the best of our knowledge, the most general
curvature bounds under which the asymptotic Dirichlet problem for the Laplace-Beltrami equation 
is solvable are given in the following theorems by Hsu (see also Theorems~\ref{HVkor1} and \ref{HVkor2}
below). 
\begin{Thm}
\cite[Theorem 1.1]{Hs}\label{HsuThm1} Let $M$ be a Cartan-Hadamard manifold.
Suppose that there exist a positive constant $a$ and a positive and
non-increasing function $h$ with $\int_{0}^{\infty}t h(t)\,dt<\infty$ such
that
\[
-h\bigl(\rho(x)\bigr)^{2}e^{2a\rho(x)}\le\ric_{x}\quad\text{and}\quad
\sect\le-a^{2}.
\]
Then the Dirichlet problem at infinity for $M$ is solvable.
\end{Thm}

\begin{Thm}
\cite[Theorem 1.2]{Hs}\label{HsuThm2} Let $M$ be a Cartan-Hadamard manifold.
Suppose that there exist positive constants $r_{0},\ \alpha>2,$ and
$\beta<\alpha-2$ such that
\[
-\rho(x)^{2\beta}\le\ric_{x}\quad\text{and}\quad\sect_{x}\le-\frac
{\alpha(\alpha-1)}{\rho(x)^{2}}
\]
for all $x\in M$, with $\rho(x)\ge r_{0}$. Then the Dirichlet problem at
infinity for $M$ is solvable.
\end{Thm}

The asymptotic Dirichlet problem has been studied also in a more general
context of $p$-harmonic and ${\mathcal{A}}$-harmonic functions as well as for
operators ${\mathcal{Q}}$. For the $p$-Laplace equation the asymptotic
Dirichlet problem was solved in \cite{Ho} on Cartan-Hadamard manifolds of
pinched negative sectional curvature by modifying the direct approach of
Anderson and Schoen \cite{andschoen}. In \cite{HoVa} Holopainen and
V\"{a}h\"{a}kangas studied the asymptotic Dirichlet problem for the
$p$-Laplace equation and the $p$-regularity of a point $x_{0}$ at infinity on
a Cartan-Hadamard manifold $M$ under a curvature assumption
\[
-b\bigl(\rho(x)\bigr)^{2}\leq\sect_{x}\leq-a\bigl(\rho(x)\bigr)^{2}%
\]
in $U\cap M$, where $U$ is a neighborhood of $x_{0}\in M(\infty)$. Here
$a,b\colon\lbrack0,\infty)\rightarrow\lbrack0,\infty),\ b\geq a,$ are smooth
functions subject to certain growth conditions; see Section~\ref{sec_preli}.
The following two special cases of functions $a$ and $b$ are of particular interest.

\begin{Thm}
\cite[Corollary 3.22]{HoVa}\label{HVkor1} Let $\phi>1$ and $\varepsilon>0$.
Let $x_{0}\in M(\infty)$ and let $U$ be a neighborhood of $x_{0}$ in the cone
topology. Suppose that
\begin{equation}
\label{kor1_kaava}-\rho(x)^{2\phi-4-\varepsilon}\le\sect_{x}\le-\frac
{\phi(\phi-1)}{\rho(x)^{2}}%
\end{equation}
for every $x\in U\cap M$. Then $x_{0}$ is a $p$-regular point at infinity for
every $p\in\bigl(1,1+(n-1)\phi\bigr)$.
\end{Thm}

\begin{Thm}
\cite[Corollary 3.23]{HoVa}\label{HVkor2} Let $k>0$ and $\varepsilon>0$. Let
$x_{0}\in M(\infty)$ and let $U$ be a neighborhood of $x_{0}$ in the cone
topology. Suppose that
\begin{equation}
\label{kor2_kaava}-\rho(x)^{-2-\varepsilon}e^{2k\rho(x)}\le\sect_{x}\le-k^{2}%
\end{equation}
for every $x\in U\cap M$. Then $x_{0}$ is a $p$-regular point at infinity for
every $p\in(1,\infty)$.
\end{Thm}

Roughly speaking, the $p$-regularity of $x_{0}\in M(\infty)$ means that, at
the point $x_{0}$, the Dirichlet problem for the $p$-Laplace equation is
solvable with continuous boundary data; see \cite{HoVa} and \cite{Va1} for the
details. In particular, the Dirichlet problem at infinity for the $p$-Laplace
equation is solvable if every point $x_{0}\in M(\infty)$ is $p$-regular.
The case of the usual Laplacian ($p=2$) is covered by Theorem~\ref{HVkor1} 
for every $\phi>1$ since then $1+(n-1)\phi>2$. Thus the assumptions in Theorem~\ref{HVkor1} 
are slightly weaker than those in Theorem~\ref{HsuThm2}. Note that using the Ricci curvature 
instead of the sectional makes no essential difference since all sectional curvatures are 
nonpositive. On the other hand, Theorem~\ref{HVkor2} and Theorem~\ref{HsuThm1} 
are closely related in the case $p=2$ but, nevertheless, slightly different and neither one 
implies the other directly.

In \cite{Va1} V\"ah\"akangas generalized the method and results due to
Cheng~\cite{cheng} and showed that $x_{0}\in M(\infty)$ is $p$-regular if it
has a neighborhood $V$ in the cone topology such that the radial sectional
curvatures in $V\cap M$ satisfy a pointwise pinching condition
\[
\lvert\sect_{x}(P) \rvert\le C\lvert\sect_{x}(P^{\prime}) \rvert
\]
for some constant $C$ and have an upper bound
\[
\sect_{x}(P) \le-\frac{\phi(\phi-1)}{\rho^{2}(x)}
\]
for some constant $\phi>1$ with $1<p<1+\phi(n-1)$. Above $P$ and $P^{\prime}$
are any $2$-dimensional subspaces of $T_{x} M$ containing the (radial) vector
$\nabla\rho(x)$. It is worth observing that no curvature lower bounds are
needed here. In fact, V\"ah\"akangas considered even a more general case of
${\mathcal{A}}$-harmonic functions (of type $p\in(1,\infty)$), i.e. continuous
weak solutions to the equation
\[
-\diver{\mathcal{A}}(\nabla u)=0,
\]
where ${\mathcal{A}}$ is subject to certain conditions; for instance
$\langle{\mathcal{A}}(V),V\rangle\approx\lvert V \rvert^{p},\ 1<p<\infty,$ and
${\mathcal{A}}(\lambda V)=\lambda\lvert\lambda\rvert^{p-2}{\mathcal{A}}(V)$
for all $\lambda\in\mathbb{R}\setminus\{0\}$. Note that this class of
equations is different from ours in the current paper, although both include
the $p$-Laplace equation. Recently, V\"ah\"akangas generalized Theorems
\ref{HVkor1} and \ref{HVkor2} to cover the case of ${\mathcal{A}}$-harmonic
functions as well; see \cite[Corollary 3.7, Corollary 3.8, Remark 3.9]{Va2}.

In \cite{CR} Collin and Rosenberg constructed harmonic diffeomorphisms from the complex plane $\C$ onto the 
hyperbolic plane $\hyp^2$ disproving a conjecture of Schoen and Yau \cite{ScYau}. A bit later G{\'a}lvez and 
Rosenberg \cite{GR} extended the result to any Hadamard surface $M$ whose curvature is bounded 
from above by a negative constant by proving the existence of harmonic diffeomorphisms from $\C$ onto $M$. 
The proofs in both papers are based on the construction of an entire minimal surface 
$\Sigma=(x,u(x))\subset \hyp^2\times \R$ ($\Sigma\subset M\times\R$, resp.) of conformal type $\C$, and thus on the 
construction of an entire solution $u$ to the minimal graph equation that is unbounded both from above and from 
below. Harmonic diffeomorphisms $\C\to \hyp^2$ ($\C\to M$, resp.) are then obtained by composing conformal mappings
(diffeomorphisms) $\C\to\Sigma$ with harmonic vertical projections $\Sigma\to\hyp^2$ ($\Sigma\to M$, resp.). 
A crucial method in the construction of an entire unbounded solution $u$ to the minimal graph equation is to solve
the Dirichlet problem on unbounded ideal polygons with boundary values $\pm\infty$ on the sides of the ideal polygons. The unexpected result of Collin and Rosenberg has raised interest 
in (entire) minimal hypersurfaces in the product space $M\times\R$, where $M$ is a Cartan-Hadamard manifold. 
Motivated by the recent research in this field (see for example, \cite{DHL}, \cite{ER}, \cite{MR}, \cite{NR}, 
\cite{RT}, \cite{RSS}, \cite{ET}, \cite{Spruck}) we investigate in the present paper a possible extension of the
results for the $p$-Laplacian obtained in \cite{HoVa} to the minimal graph PDE.

Of particular interest is the following special case of our main theorem
(Theorem~\ref{ThmMain}).

\begin{Thm}
\label{thm1} Let $M$ be a Cartan-Hadamard manifold of dimension $n\ge2$. Fix
$o\in M$ and set $\rho(\cdot)=d(o,\cdot)$, where $d$ is the Riemannian
distance in $M$. Assume that
\begin{equation}
\label{curv_ass_minim}-\rho(x)^{2\left(  \phi-2\right)  -\varepsilon}%
\leq\sect_{x}(P)\leq-\dfrac{\phi(\phi-1)}{\rho(x)^{2}},
\end{equation}
for some constants $\phi>1$ and $\varepsilon>0,$ where $\sect_{x}(P)$ is the
sectional curvature of a plane $P\subset T_{x}M$ and $x$ is any point in the
complement of a ball $B(o,R_{0})$. Then the asymptotic Dirichlet problem for
the minimal graph equation \eqref{min_eqn} is uniquely solvable for any
boundary data $f\in C\bigl(M(\infty)\bigr)$.
\end{Thm}

So far, the solvability of the asymptotic Dirichlet problem for the minimal graph equation has been
established only under hypothesis which included the condition 
$\sect_{x}(P)\leq c<0$ (see \cite{GR}, \cite{RT}). In \cite{RT} Ripoll and Telichevesky 
introduced the following \emph{strict convexity condition\/} (SC condition) that applies to equations 
\eqref{M_equ}. A Cartan-Hadamard manifold $M$ satisfies the strict convexity condition if, for every $x\in M(\infty)$
and relatively open subset $W\subset M(\infty)$ containing $x$, there exists a $C^2$ open subset 
$\Omega\subset M$ such that $x\in\Int(M(\infty))\subset W$ and $M\setminus\Omega$ is convex.  
They proved that the asymptotic Dirichlet problem for \eqref{M_equ} on $M$ is solvable if $\sect\le -k^2<0$
and $M$ satisfies the SC condition; see \cite[Theorem 7]{RT}. Furthermore, they showed by modifying 
Anderson's and Borb\'{e}ly's arguments 
that $M$ satisfies the SC condition provided there exist constants $k>0,\ \varepsilon>0$,
and $R^{\ast}$ such that
\[
-\rho(x)^{-2-\varepsilon}e^{2k\rho(x)}\le \sect_{x}\le -k^2
\]
for all $x\in M\setminus B(o,R^{\ast})$ thus generalizing Theorem~\ref{HVkor2}; see \cite[Theorem 14]{RT}.

The main theorem of the paper is the following solvability result for the
asymptotic Dirichlet problem for operators ${\mathcal{Q}}$ that satisfy
\eqref{Agrowth}, \eqref{Bgrowth}, and conditions (A) and (B) under curvature
assumption
\[
-b\bigl(\rho(x)\bigr)^{2}\leq\sect_{x}\leq-a\bigl(\rho(x)\bigr)^{2}%
\]
on $M$, where $a,b\colon\lbrack0,\infty)\rightarrow\lbrack0,\infty),\ b\geq
a,$ are smooth functions satisfying assumptions \eqref{A1}-\eqref{A7} (see
Section~\ref{sec_preli}). The constant $\phi_{1}$ below is related to the
assumption~\eqref{A1}. More precisely,
\[
\phi_{1}=\frac{1+\sqrt{1+4C_{1}^{2}}}{2}>1,
\]
where $C_{1}>0$ is a constant such that, for all $t\geq T_{1}>0$,
\[
a(t)%
\begin{cases}
=C_{1}t^{-1} & \text{if $b$ is decreasing,}\\
\geq C_{1}t^{-1} & \text{if $b$ is increasing.}
\end{cases}
\]
We also recall that $B_{0}$ is the constant in the assumption \eqref{Bgrowth}.

\begin{Thm}
\label{ThmMain} Let $M$ be a Cartan-Hadamard manifold of dimension $n\ge2$.
Fix $o\in M$ and set $\rho(\cdot)=d(o,\cdot)$, where $d$ is the Riemannian
distance in $M$. Assume that
\begin{equation}
\label{curv_assump_gen}-(b\circ\rho)^{2}(x)\le\sect_{x}(P)\le-(a\circ\rho
)^{2}(x)
\end{equation}
for all $x\in M$ and all $2$-dimensional subspaces $P\subset T_{x}M$. Then the
asymptotic Dirichlet problem for the equation \eqref{M_equ} is uniquely
solvable for any boundary data $f\in C\bigl(M(\infty)\bigr)$ whenever
$B_{0}<\tfrac12((n-1)\phi_{1}-1)$.
\end{Thm}

Observe that $B_{0}=0$ for the minimal graph equation ${\mathcal{M}}[u]=0$,
and therefore the condition $B_{0}<\tfrac12((n-1)\phi_{1}-1)$ is satisfied in
Theorem~\ref{thm1}. On the other hand, in the case of the $p$-Laplacian this
condition reads as $1<p<(n-1)\phi+1$ and it is known to be sharp; see
\cite[Example 2]{Va1}.

Another special case, where the curvature is bounded from above by a negative
constant $-k^{2}$, generalizes Theorem~\ref{HVkor2} and gives another proof
for the above mentioned result of Ripoll and Telichevesky \cite[Theorem 14]{RT}. Here
no further restriction for the constant $B_{0}$ is needed. We refer to
Examples~\ref{ex1} and \ref{ex2} for the verification of the assumptions
\eqref{A1}-\eqref{A7} for the curvature bounds in Theorem~\ref{thm1} and
Corollary~\ref{HVkor2_RT}.

\begin{corollary}
\label{HVkor2_RT} Let $M$ be a Cartan-Hadamard manifold of dimension $n\ge2$.
Fix $o\in M$ and set $\rho(\cdot)=d(o,\cdot)$, where $d$ is the Riemannian
distance in $M$. Assume that
\begin{equation}
\label{curv_assump_k}-\rho(x)^{-2-\varepsilon}e^{2k\rho(x)}\le\sect_{x}%
(P)\le-k^{2}%
\end{equation}
for some constants $k>0$ and $\varepsilon>0$ and for all $x\in M\setminus
B(o,R_{0})$. Then the asymptotic Dirichlet problem for the equation
\eqref{M_equ} is uniquely solvable for any boundary data $f\in C\bigl(M(\infty
)\bigr)$.
\end{corollary}

We close this introduction with comments on the necessity of
curvature bounds. It is worth of pointing out that the curvature
bounds used in this paper are essentially the most general ones under which
the asymptotic Dirichlet problem is known to be solvable, for instance, for
the usual Laplace equation (\cite{Hs}), for the $p$-Laplace equation or the
${\mathcal{A}}$-harmonic equation (\cite{HoVa}, \cite{Va2}), or for the
minimal graph equation (\cite{RT} and the current paper). On the other hand, Ancona's and
Borb\'ely's examples (\cite{ancrevista}, \cite{Bor}) show that a (strictly)
negative curvature upper bound alone is not sufficient for the solvability of
the asymptotic Dirichlet problem for the Laplace equation. In \cite{H_ns},
Holopainen generalized Borb\'ely's result to the $p$-Laplace equations, and
very recently, Holopainen and Ripoll \cite{HR_ns} extended these
nonsolvability results to equations \eqref{M_equ}, in particular, to the
minimal graph equation.


\section{Preliminaries}

\label{sec_preli}
In this section we introduce the assumptions for the curvature bounds and
consider the settings in Theorem~\ref{thm1} and Corollary \ref{HVkor2_RT} as examples.

We start with the following \emph{Comparison principle\/} that is crucial for
the rest of the paper. Although its short proof follows the ideas in
\cite[Lemma 3.18]{HKM} (see also \cite[Lemma 3]{RT}) we feel it appropriate to
give the details.

\begin{Lem}
\label{lem_comp} If $u\in W^{1,p}(\Omega)$ is a ${\mathcal{Q}}$-supersolution
and $v\in W^{1,p}(\Omega)$ is a ${\mathcal{Q}}$-subsolution such that
$\varphi=\min(u-v,0)\in W^{1,p}_{0}(\Omega)$, then $u\ge v$ a.e. in $\Omega$.
\end{Lem}

\begin{proof}
Using the non-negative function $-\varphi$ as a test function we obtain
\begin{align*}
0  &  \ge\int_{\Omega}\bigl\langle{\mathcal{A}}\bigl(\lvert\nabla v \rvert
^{2}\bigr)\nabla v,-\nabla\varphi\bigr\rangle\,dm -\int_{\Omega}%
\bigl\langle{\mathcal{A}}\bigl(\lvert\nabla u \rvert^{2}\bigr)\nabla
u,-\nabla\varphi\bigr\rangle\,dm\\
&  = \int_{\Omega\cap\{u<v\}}\bigl\langle{\mathcal{A}}\bigl(\lvert\nabla v
\rvert^{2}\bigr)\nabla v -{\mathcal{A}}\bigl(\lvert\nabla u \rvert
^{2}\bigr)\nabla u,\nabla v-\nabla u\bigr\rangle\,dm.
\end{align*}
On the other hand, estimating the integrand from below by the Cauchy-Schwarz
inequality we obtain
\begin{align*}
&  \bigl\langle{\mathcal{A}}\bigl(\lvert\nabla v \rvert^{2}\bigr)\nabla v
-{\mathcal{A}}\bigl(\lvert\nabla u \rvert^{2}\bigr)\nabla u,\nabla v-\nabla
u\bigr\rangle\\
\ge\,  &  {\mathcal{A}}\bigl(\lvert\nabla v \rvert^{2}\bigr)\lvert\nabla v
\rvert^{2}-{\mathcal{A}}\bigl(\lvert\nabla v \rvert^{2}\bigr)\lvert\nabla v
\rvert\lvert\nabla u \rvert-{\mathcal{A}}\bigl(\lvert\nabla u \rvert
^{2}\bigr)\lvert\nabla u \rvert\lvert\nabla v \rvert+ {\mathcal{A}%
}\bigl(\lvert\nabla u \rvert^{2}\bigr)\lvert\nabla u \rvert^{2}\\
=\,  &  \left(  \lvert\nabla v \rvert{\mathcal{A}}\bigl(\lvert\nabla v
\rvert^{2}\bigr)-\lvert\nabla u \rvert{\mathcal{A}}\bigl(\lvert\nabla u
\rvert^{2}\bigr) \right)  \left(  \lvert\nabla v \rvert-\lvert\nabla u
\rvert\right)  \ge0,
\end{align*}
where the last inequality holds since $t\mapsto t{\mathcal{A}}(t^{2})$ is
increasing. Hence the non-negative integrand must vanish a.e. in $\Omega
\cap\{u<v\}$. Furthermore, since $t\mapsto t{\mathcal{A}}(t^{2})$ is strictly
increasing, we have $\lvert\nabla u \rvert=\lvert\nabla v \rvert$ a.e. in
$\Omega\cap\{u<v\}$, but then
\[
0 = \bigl\langle{\mathcal{A}}\bigl(\lvert\nabla v \rvert^{2}\bigr)\nabla v
-{\mathcal{A}}\bigl(\lvert\nabla u \rvert^{2}\bigr)\nabla u,\nabla v-\nabla
u\bigr\rangle
= {\mathcal{A}}(\bigl(\lvert\nabla v \rvert^{2}\bigr)\lvert\nabla v-\nabla u
\rvert^{2}
\]
a.e. in $\Omega\cap\{u<v\}$, and so $\nabla\varphi=0$ a.e. in $\Omega
\cap\{u<v\}$. Because $\varphi\in W^{1,p}_{0}(\Omega)$, we finally have
$\varphi=0$ a.e. in $\Omega$ and the claim follows.
\end{proof}

As a consequence, we obtain the uniqueness of ${\mathcal{Q}}$-solutions with
fixed (Sobolev) boundary data.

\begin{corollary}
\label{cor_uniq} If $u\in W^{1,p}(\Omega)$ and $v\in W^{1,p}(\Omega)$ are
${\mathcal{Q}}$-solutions with $u-v\in W^{1,p}_{0}(\Omega)$, then $u=v$ a.e.
in $\Omega$.
\end{corollary}

We will use extensively various estimates obtained in \cite{HoVa} (and 
originated in the unpublished licentiate thesis \cite{Va_lic}). Therefore
for readers' convenience we use basically the same notation
as in \cite{HoVa}. Thus we let $M$ be a Cartan-Hadamard manifold, $M(\infty)$
the sphere at infinity, and $\bar M=M\cup M(\infty)$. Recall that the sphere
at infinity is defined as the set of all equivalence classes of unit speed
geodesic rays in $M$; two such rays $\gamma_{1}$ and $\gamma_{2}$ are
equivalent if $\sup_{t\ge0}d\bigl(\gamma_{1}(t),\gamma_{2}(t)\bigr)< \infty$.
For each $x\in M$ and $y\in\bar M\setminus\{x\}$ there exists a unique unit
speed geodesic $\gamma^{x,y}\colon\mathbb{R}\to M$ such that $\gamma^{x,y}%
_{0}=x$ and $\gamma^{x,y}_{t}=y$ for some $t\in(0,\infty]$. If $v\in
T_{x}M\setminus\{0\}$, $\alpha>0$, and $r>0$, we define a cone
\[
C(v,\alpha)=\{y\in\bar M\setminus\{x\}:\sphericalangle(v,\dot\gamma^{x,y}%
_{0})<\alpha\}
\]
and a truncated cone
\[
T(v,\alpha,r)=C(v,\alpha)\setminus\bar B(x,r),
\]
where $\sphericalangle(v,\dot\gamma^{x,y}_{0})$ is the angle between vectors
$v$ and $\dot\gamma^{x,y}_{0}$ in $T_{x} M$. All cones and open balls in $M$
form a basis for the cone topology on $\bar M$.

Throughout the paper we assume that sectional curvatures of $M$ are bounded
both from above and below by
\begin{equation}
\label{curv_assump}-(b\circ\rho)^{2}(x)\le\sect_{x}(P)\le-(a\circ\rho
)^{2}(x),\ \rho(x)=d(x,o),
\end{equation}
for all $x\in M$ and all $2$-dimensional subspaces $P\subset T_{x}M$. Here $a$
and $b$ are smooth functions $[0,\infty)\to[0,\infty)$ that are constant in
some neighborhood of $0$ and $b\ge a$. Furthermore, we assume that $b$ is
monotonic and that there exist constants $T_{1},C_{1},C_{2},C_{3}>0$, and
$Q\in(0,1)$ such that
\begin{align}
\label{A1}a(t)%
\begin{cases}
=C_{1}t^{-1} & \text{if $b$ is decreasing,}\\
\ge C_{1}t^{-1} & \text{if $b$ is increasing}
\end{cases}
\tag{A1}%
\end{align}
for all $t\ge T_{1}$ and
\begin{align}
\label{A5}a(t)  &  \le C_{2},\tag{A2}\\
b(t+1)  &  \le C_{2}b(t),\tag{A3}\\
b(t/2)  &  \le C_{2}b(t),\tag{A4}\\
b(t)  &  \ge C_{3}(1+t)^{-Q} \tag{A5}%
\end{align}
for all $t\ge0$. In addition, we assume that
\begin{align}
\label{A6} &  \lim_{t\to\infty}\frac{b^{\prime}(t)}{b(t)^{2}}=0 \tag{A6}%
\end{align}
and that there exists a constant $C_{4}>0$ such that
\begin{align}
\label{A7} &  \lim_{t\to\infty}\frac{t^{1+C_{4}}b(t)}{f_{a}^{\prime}(t)}=0.
\tag{A7}%
\end{align}

The curvature bounds are needed to control first and second order derivatives
of certain ''barrier" functions that will be constructed in the next section.
To this end, if $k\colon[0,\infty)\to[0,\infty)$ is a smooth function, we
denote by $f_{k}\in C^{\infty}\bigl([0,\infty)\bigr)$ the solution to the
initial value problem
\begin{equation}
\label{Jacobi_eq}\left\{
\begin{aligned} f_k(0)&=0, \\ f_k'(0)&=1, \\ f_k''&=k^2f_k. \end{aligned} \right.
\end{equation}
It follows that the solution $f_{k}$ is a non-negative smooth function.

We close this section with two examples where we verify that the curvature
bounds that appear in Theorem~\ref{thm1} and Corollary~\ref{HVkor2_RT} satisfy
the assumption \eqref{A1}-\eqref{A7}.

\begin{example}
\label{ex1} As a first example we consider the curvature bounds in
Theorem~\ref{thm1}. Write $C_{1}=\sqrt{\phi(\phi-1)}$. We may assume that
$\varepsilon<2\phi-2$. For $t\ge R_{0}$ let
\[
a(t)=\frac{C_{1}}{t}
\]
and
\[
b(t)=t^{\phi-2-\varepsilon/2}
\]
and extend them to smooth functions $a\colon[0,\infty)$ and $b\colon
[0,\infty)\to(0,\infty)$ such that they are constants in some
neighborhood of $0$, $b$ is monotonic and $b\ge a$. This is possible since
\[
C_{1}t^{-1}\le t^{\phi-2-\varepsilon/2}
\]
for $t\ge R_{0}$ by the curvature assumption \eqref{curv_ass_minim}. It is
easy to verify that then
\[
f_{a}(t)=c_{1}t^{\phi_{1}}+c_{2}t^{1-\phi_{1}}
\]
for all $t\ge R_{0}$, where
\[
\phi_{1}=\frac{1+\sqrt{1+4C_{1}^{2}}}{2}>1,
\]
\[
c_{1}=R_{0}^{-\phi_{1}}\frac{f_{a}(R_{0})(\phi_{1}-1)+R_{0}f_{a}^{\prime
}(R_{0})}{2\phi_{1}-1}>0,
\]
and
\[
c_{2}=R_{0}^{\phi_{1}-1}\frac{f_{a}(R_{0})\phi_{1}-R_{0}f_{a}^{\prime}(R_{0}%
)}{2\phi_{1}-1}.
\]
We then have
\[
\lim_{t\to\infty}\frac{tf_{a}^{\prime}(t)}{f_{a}(t)}=\phi_{1}
\]
and, for all $C_{4}\in(0,\varepsilon/2)$
\[
\lim_{t\to\infty}\frac{t^{1+C_{4}}b(t)}{f_{a}^{\prime}(t)}=0.
\]
It follows that $a$ and $b$ satisfy \eqref{A1}-\eqref{A7} with constants
$T_{1}=R_{0}$, $C_{1}$, some $C_{2}>0$, some $C_{3}>0$, $Q=\max\{1/2,-\phi
+2+\varepsilon/2\}$, and any $C_{4}\in(0,\varepsilon/2)$.
\end{example}

\begin{example}
\label{ex2} Let $k>0$ and $\varepsilon>0$ be constants and define $a(t)=k$ for
all $t\ge0$. Define
\[
b(t)=t^{-1-\varepsilon/2}e^{kt}
\]
for $t\ge R_{0}=r_{0}+1$, where $r_{0}>0$ is so large that $t\mapsto
t^{-1-\varepsilon/2}e^{kt}$ is increasing and greater than $k$ for all $t\ge
r_{0}$. Extend $b$ to an increasing smooth function $b\colon[0,\infty
)\to[k,\infty)$ that is constant in some neighborhood of $0$. Now we can
choose $C_{1}>0$ in \eqref{A1} as large as we wish. In particular, once the
operator ${\mathcal{A}}$ and hence the constant $B_{0}$ is chosen, we may fix
$C_{1}$ so large that
\[
\phi_{1}=\frac{1+\sqrt{1+4C_{1}^{2}}}{2}
\]
satisfies $B_{0}<\tfrac12 ((n-1)\phi_{1}-1)$. Then $a$ and $b$ satisfy
\eqref{A1}-\eqref{A7} with constants $C_{1},\ T_{1}=C_{1}/k$, some $C_{2}>0$,
some $C_{3}>0$, $Q=1/2$, and any $C_{4}\in(0,\varepsilon/2)$.
\end{example}


\section{Construction of a barrier}

\label{sec_barrier}
To solve the asymptotic Dirichlet problem for ${\mathcal{Q}}$ with given
continuous boundary data $f\in C\bigl(M(\infty)\bigr)$, the first task is to
construct a ''barrier" for each boundary point $x_{0}\in M(\infty)$. For that
purpose let $v_{0}=\dot\gamma^{o,x_{0}}_{0}$ be the initial (unit) vector of
the geodesic ray $\gamma^{o,x_{0}}$ from a fixed point $o\in M$ and define a
function $h:M(\infty)\to\mathbb{R}$,
\begin{equation}
\label{eq:hoodef}h(x)=\min\bigl(1,L\sphericalangle(v_{0},\dot\gamma^{o,x}%
_{0})\bigr),
\end{equation}
where $L\in(8/\pi,\infty)$ is a constant.

Next step is to extend $h$ to a function $h\in C^{\infty}(M)\cap C(\bar M)$
with controlled first and second order derivatives. This is done in
\cite{HoVa} by defining first a crude extension $\tilde h:\bar M\to\mathbb{R}%
$,
\begin{equation}
\label{eq:hoodeftilde}\tilde h(x)=\min\Bigl(1,\max\bigl(2-2\rho
(x),L\sphericalangle(v_{0},\dot\gamma^{o,x}_{0})\bigr)\Bigr).
\end{equation}
Then $\tilde h\in C(\bar M)$ and $\tilde h|M(\infty)=h$. As the final step in
the construction of a barrier we smooth out $\tilde{h}$ to get an extension
$h\in C^{\infty}(M)\cap C(\bar M)$. To this end, we fix $\chi\in C^{\infty
}(\mathbb{R})$ such that $0\le\chi\le1$, $\spt\chi\subset[-2,2]$, and
$\chi\vert[-1,1]\equiv1$. Then for any function $\varphi\in C(M)$ we define
functions $F_{\varphi}\colon M\times M\to\mathbb{R},\ {\mathcal{R}}%
(\varphi)\colon M\to M$, and ${\mathcal{P}}(\varphi)\colon M\to\mathbb{R}$ by
\begin{align*}
F_{\varphi}(x,y)  &  =\chi\bigl(b(\rho(y))d(x,y)\bigr)\varphi(y),\\
{\mathcal{R}}(\varphi)(x)  &  =\int_{M}F_{\varphi}(x,y) dm(y),\ \text{ and}\\
{\mathcal{P}}(\varphi)  &  =\frac{{\mathcal{R}}(\varphi)}{{\mathcal{R}}(1)},
\end{align*}
where
\[
{\mathcal{R}}(1)=\int_{M}\chi\bigl(b(\rho(y))d(x,y)\bigr)dm(y)>0.
\]
If $\varphi\in C(\bar M)$, we extend ${\mathcal{P}}(\varphi)\colon
M\to\mathbb{R}$ to a function $\bar{M}\to\mathbb{R}$ by setting ${\mathcal{P}%
}(\varphi)(x)=\varphi(x)$ whenever $x\in M(\infty)$. Then the extended
function ${\mathcal{P}}(\varphi)$ is $C^{\infty}$-smooth in $M$ and continuous
in $\bar{M}$; see \cite[Lemma 3.13]{HoVa}. In particular, applying
${\mathcal{P}}$ to the function $\tilde{h}$ yields an appropriate smooth
extension
\begin{equation}
\label{extend_h}h:={\mathcal{P}}(\tilde{h})
\end{equation}
of the original function $h\in C\bigl(M(\infty)\bigr)$ that was defined in \eqref{eq:hoodef}.

We obtain control on first and second order derivatives of the extended
function $h$ from the curvature assumption \eqref{curv_assump} by the Rauch
and Hessian comparison theorems. Here the solutions $f_{a}$ and $f_{b}$ to the
initial value problem \eqref{Jacobi_eq}, where $a$ and $b$ are curvature
bounds in \eqref{curv_assump} satisfying \eqref{A1}-\eqref{A7}, play an important
role. Another crucial point is that the mollifying procedure above depends on the curvature lower bound
function $b$.
For the next lemma and later purposes we denote
\[
\Omega=C(v_{0},1/L)\cap M \ \text{ and }\ k\Omega=C(v_{0},k/L)\cap M
\]
for $k>0$. We collect various constants and functions together to a data
\[
C=(a,b,T_{1},C_{1},C_{2},C_{3},C_{4},Q,n,L).
\]
Furthermore, we denote by $\|\Hess_{x} u\|$ the norm of the Hessian of a
smooth function $u$ at $x$, that is
\[
\|\Hess_{x} u\|=\sup_{ \overset{ \mbox{\scriptsize$X\in T_xM$} }{\lvert X
\rvert\le1}}\lvert\Hess u(X,X) \rvert.
\]
Our (first) main estimates are the following.

\begin{Lem}
\cite[Lemma 3.16]{HoVa}\label{arvio_lause} There exist constants $R_{1}%
=R_{1}(C)$ and $c_{5}=c_{5}(C)$ such that the extended function $h\in
C^{\infty}(M)\cap C(\bar M)$ in \eqref{extend_h} satisfies
\begin{equation}
\label{arvio1}%
\begin{split}
|\nabla h(x)|  &  \le c_{5}\frac{1}{(f_{a}\circ\rho)(x)},\\
\|\Hess_{x} h\|  &  \le c_{5}\frac{(b\circ\rho)(x)}{(f_{a}\circ\rho)(x)},\\
\end{split}
\end{equation}
for all $x\in3\Omega\setminus B(o,R_{1})$. In addition,
\[
h(x)=1
\]
for every $x\in M\setminus\bigl(2\Omega\cup B(o,R_{1})\bigr)$.
\end{Lem}

Let then $A>0$ be a fixed constant. We aim to show that
\begin{equation}
\label{varphi_def}\varphi=A(R_{4}^{\delta}\rho^{-\delta}+h)
\end{equation}
is a ${\mathcal{Q}}$-supersolution in the set $3\Omega\setminus\bar{B}%
(o,R_{4})$, where $\delta>0$ and $R_{4}>0$ are constants that will be
specified later and $h$ is the extended function defined in \eqref{extend_h}.
First of all $\varphi$ is $C^{\infty}$-smooth in $M\setminus\{o\}$ and there
\[
\nabla\varphi=A\bigl(-R_{4}^{\delta}\delta\rho^{-\delta-1}\nabla\rho+\nabla
h\bigr)
\]
and
\[
\Delta\varphi=A\bigl(R_{4}^{\delta}\delta(\delta+1)\rho^{-\delta-2}%
-R_{4}^{\delta}\delta\rho^{-\delta-1}\Delta\rho+\Delta h\bigr).
\]
We shall make use of the following estimates obtained in \cite{HoVa}; see also \cite{Ho}:

\begin{Lem}
\cite[Lemma 3.17]{HoVa}\label{perusta} There exist constants $R_{2}=R_{2}(C)$
and $c_{6}=c_{6}(C)$ with the following property. If $\delta\in(0,1)$, then
\[%
\begin{split}
|\nabla h|  &  \le c_{6}/(f_{a}\circ\rho),\\
\|\Hess h\|  &  \le c_{6}\rho^{-C_{4}-1}(f_{a}^{\prime}\circ\rho)/(f_{a}%
\circ\rho),\\
|\nabla\langle\nabla h,\nabla h\rangle|  &  \le c_{6}\rho^{-C_{4}-2}%
(f_{a}^{\prime}\circ\rho)/(f_{a}\circ\rho),\\
|\nabla\langle\nabla h,\nabla(\rho^{-\delta})\rangle|  &  \le c_{6}%
\rho^{-C_{4}-2}(f_{a}^{\prime}\circ\rho)/(f_{a}\circ\rho),\\
\nabla\bigl\langle\nabla(\rho^{-\delta}),\nabla(\rho^{-\delta})\bigr\rangle
&  =-2\delta^{2}(\delta+1)\rho^{-2\delta-3}\nabla\rho
\end{split}
\]
in the set $3\Omega\setminus B(o,R_{2})$.
\end{Lem}

As in \cite{HoVa} we denote
\[
\phi_{1}=\frac{1+\sqrt{1+4C_{1}^{2}}}{2}>1,\quad\text{and}\quad\delta_{1}%
=\min\left\lbrace C_{4},\frac{-1+(n-1)\phi_{1}}{1+(n-1)\phi_{1}}\right\rbrace
\in(0,1),
\]
where $C_{1}$ and $C_{4}$ are constants from \eqref{A1} and \eqref{A7},
respectively. Then by \cite[Lemma 3.18]{HoVa} there exists $R_{3}%
=R_{3}(C,\delta)$ such that
\begin{equation}
\label{2nd_estim}-\Delta(\rho^{-\delta})>0 \quad\text{and}\quad\frac
{\lvert\Delta h \rvert}{-\Delta(\rho^{-\delta})}\le\delta
\end{equation}
in $3\Omega\setminus B(o,R_{3})$.

Suppose then that
\[
B_{0}<\tfrac12\bigl((n-1)\phi_{1} -1\bigr)
\]
and let $0<\delta<\min(\delta_{1},\phi_{1}-1,C_{4}/2)$ be so small that
\begin{equation}
\label{delta_bound}\delta+ \frac{2\lambda\bigl(\max(0,B_{0})+\bar{B_{0}}%
\delta\bigr)}{(1-\lambda)(1-\delta)^{3}} <1,
\end{equation}
where $B_{0}$ is the constant in \eqref{Bgrowth}, $\bar{B_{0}}=\max
(\tfrac12,B_{0})$, and
\[
\lambda=\frac{1+\delta}{(1-\delta)(n-1)\phi_{1}}\in(0,1).
\]
Such $\delta$ exists because $B_{0}<\tfrac12((n-1)\phi_{1} -1)$. Then there
exists $R_{4}=R_{4}(C,B_{0})\ge\min(R_{3},1)$ such that, in addition to
estimates in \eqref{2nd_estim}, we have
\begin{equation}
\label{(3.30)}\frac{-\Delta(\rho^{-\delta})}{\delta\rho^{-\delta-1}\Delta\rho
}\ge1-\lambda, \quad\frac{\lvert\nabla h \rvert}{\lvert\nabla(\rho^{-\delta})
\rvert}\le\delta,\quad\frac{\rho(f_{a}^{\prime}\circ\rho)}{f_{a}\circ\rho}%
\ge(1-\delta)\phi_{1},
\end{equation}
and
\begin{equation}
\label{(3.33)}\frac{3\bar{B_{0}}c_{6}\rho^{-C_{4}+2\delta}}{R_{4}^{\delta
}(1-\delta/R_{4}^{\delta})^{2}\delta^{2}(1-\lambda)(n-1)} \le\delta
\end{equation}
in $3\Omega\setminus B(o,R_{4})$; see \cite[Lemma 3.18, (3.30), (3.32)]{HoVa}
for the estimates in \eqref{(3.30)}. The estimate \eqref{(3.33)} is possible
because $-C_{4}+2\delta<0$.

We are now in a position to prove that $\varphi=A(R_{4}^{\delta}\rho^{-\delta
}+h)$, for any given constant $A>0,$ is a ${\mathcal{Q}}$-supersolution in an
open truncated cone $3\Omega\setminus\bar{B}(o,R_{4})$ whenever $B_{0}%
<\tfrac12((n-1)\phi_{1}-1)$. As a smooth function $\varphi$ is a
${\mathcal{Q}}$-supersolution if $\diver{\mathcal{A}}(\lvert\nabla
\varphi\rvert^{2})\nabla\varphi\le0$. On the other hand, it follows from
\eqref{2nd_estim} and \eqref{(3.30)} that
\[
\lvert\nabla\varphi\rvert>0\quad\text{and}\quad\Delta\varphi<0
\]
in $3\Omega\setminus B(o,R_{4})$. Hence we can write
\begin{align*}
\diver{\mathcal A}(\lvert\nabla\varphi\rvert^{2})\nabla\varphi &
={\mathcal{A}}(\lvert\nabla\varphi\rvert)^{2}\Delta\varphi+ {\mathcal{A}%
}^{\prime2})\left\langle \nabla\langle\nabla\varphi,\nabla\varphi\rangle,
\nabla\varphi\right\rangle \\
&  ={\mathcal{A}}(\lvert\nabla\varphi\rvert)^{2}\left\lbrace \Delta\varphi+
{\mathcal{B}}(\lvert\nabla\varphi\rvert^{2}) \left\langle \nabla\langle
\nabla\varphi,\nabla\varphi\rangle, \nabla\varphi\right\rangle \right\rbrace
\\
\end{align*}
in $3\Omega\setminus B(o,R_{4})$. Therefore $\varphi$ is a ${\mathcal{Q}}%
$-supersolution if
\begin{equation}
\label{Qsuper_cond}\frac{{\mathcal{B}}(\lvert\nabla\varphi\rvert^{2}%
)\lvert\nabla\varphi\rvert^{2}\left\langle \nabla\langle\nabla\varphi
,\nabla\varphi\rangle, \nabla\varphi\right\rangle }{-\lvert\nabla\varphi
\rvert^{2}\Delta\varphi} < 1
\end{equation}
in $3\Omega\setminus\bar{B}(o,R_{4})$.

\begin{Lem}
\label{supsol} Let $A>0$ be a fixed constant and let $h$ be the function
defined in \eqref{extend_h}. Then there exist constants $\delta=\delta
(C,B_{0})\in(0,\delta_{1})$ and $R_{4}=R_{4}(C,B_{0})$ such that the function
$\varphi=A(R_{4}^{\delta}\rho^{-\delta}+h)$ is a ${\mathcal{Q}}$-supersolution
in the set $3\Omega\setminus\bar{B}(o,R_{4})$ whenever $B_{0}<\tfrac
12((n-1)\phi_{1}-1)$.
\end{Lem}

\begin{proof}
Since all estimates in this proof are made in the set $3\Omega\setminus
B(o,R_{4})$, we do not indicate this all the time. Writing $u=R_{4}^{\delta
}\rho^{-\delta}+h$ we have
\[%
\begin{split}
&  \frac{{\mathcal{B}}(\lvert\nabla\varphi\rvert^{2})\lvert\nabla\varphi
\rvert^{2}\bigl\langle\nabla\langle\nabla\varphi,\nabla\varphi\rangle,
\nabla\varphi\bigr\rangle}{-\lvert\nabla\varphi\rvert^{2}\Delta\varphi}
=\frac{{\mathcal{B}}(\lvert\nabla\varphi\rvert^{2})\lvert\nabla\varphi
\rvert^{2}\bigl\langle\nabla\langle\nabla u,\nabla u\rangle,\nabla
u\bigr\rangle}{-\lvert\nabla u \rvert^{2}\Delta u}\\
&  =\frac{{\mathcal{B}}(\lvert\nabla\varphi\rvert^{2})\lvert\nabla
\varphi\rvert^{2}}{-\lvert\nabla u \rvert^{2}\Delta u} \Bigl( R_{4}^{3\delta
}\bigl\langle\nabla\langle\nabla(\rho^{-\delta}),\nabla(\rho^{-\delta}%
)\rangle,\nabla(\rho^{-\delta})\bigr\rangle\\
&  \quad+ \bigl\langle\nabla\langle\nabla h,\nabla h\rangle+2R_{4}^{\delta
}\nabla\langle\nabla h,\nabla(\rho^{-\delta})\rangle,\nabla
u\bigr\rangle +R_{4}^{2\delta}\bigl\langle\nabla\langle\nabla(\rho^{-\delta
}),\nabla(\rho^{-\delta})\rangle,\nabla h\bigr\rangle\Bigr)\\
&  \le\frac{{\mathcal{B}}(\lvert\nabla\varphi\rvert^{2})\lvert\nabla
\varphi\rvert^{2} R_{4}^{3\delta} \bigl\langle\nabla\langle\nabla
(\rho^{-\delta}),\nabla(\rho^{-\delta})\rangle,\nabla(\rho^{-\delta
})\bigr\rangle} {-\lvert\nabla u \rvert^{2}\Delta u}\\
&  \quad+ \frac{\bar{B_{0}}\left(  \bigl|\nabla\langle\nabla h,\nabla
h\rangle\bigr|+2R_{4}^{\delta}\bigl|\nabla\langle\nabla h,\nabla(\rho
^{-\delta})\rangle\bigr|\right)  }{-\lvert\nabla u \rvert\Delta u} +
\frac{\bar{B_{0}}R_{4}^{2\delta}\bigl|\nabla\langle\nabla(\rho^{-\delta
}),\nabla(\rho^{-\delta})\rangle\bigr| \lvert\nabla h \rvert}{-\lvert\nabla u
\rvert^{2}\Delta u}.\\
\end{split}
\]
We estimate the three terms above separately. By the standard Laplace
comparison (see e.g. \cite[Prop. 2.5(b)]{HoVa}) and \eqref{(3.30)} we have
\begin{equation}
\label{lap_rho_est}\Delta\rho\ge(n-1)\frac{f_{a}^{\prime}\circ\rho}{f_{a}%
\circ\rho}\ge\frac{(n-1)(1-\delta)\phi_{1}}{\rho}.
\end{equation}
As in \cite{HoVa}, we denote
\[
T=\frac{\lvert\nabla\langle\nabla(\rho^{-\delta}),\nabla(\rho^{-\delta
})\rangle\rvert}{-\lvert\nabla u \rvert\Delta u} =\frac{2\delta^{2}
(\delta+1)\rho^{-2\delta-3}}{-\lvert\nabla u \rvert\Delta u}.
\]
Using \eqref{2nd_estim}, \eqref{(3.30)}, and \eqref{lap_rho_est}, we first
obtain
\begin{equation}
\label{est_T_denom}%
\begin{split}
-\lvert\nabla u \rvert\Delta u  &  \ge-(R_{4}^{\delta}-\delta)^{2}
\lvert\nabla(\rho^{-\delta}) \rvert\Delta(\rho^{-\delta})\\
&  \ge(R_{4}^{\delta}-\delta)^{2}\delta^{2}(1-\lambda)\rho^{-2\delta-2}%
\Delta\rho\\
&  \ge(R_{4}^{\delta}-\delta)^{2}\delta^{2}(1-\lambda)\rho^{-2\delta
-2}(n-1)\frac{f_{a}^{\prime}\circ\rho}{f_{a}\circ\rho}\\
&  \ge(R_{4}^{\delta}-\delta)^{2}\delta^{2}(1-\lambda)\rho^{-2\delta
-3}(n-1)(1-\delta)\phi_{1},
\end{split}
\end{equation}
and therefore
\begin{equation}
\label{T_est}T=\frac{2\delta^{2}(\delta+1)\rho^{-2\delta-3}}{-\lvert\nabla u
\rvert\Delta u} \le\frac{2\lambda}{(R_{4}^{\delta}-\delta)^{2}(1-\lambda)}.
\end{equation}
Since
\[
\frac{\bigl\langle\nabla\langle\nabla(\rho^{-\delta}),\nabla(\rho^{-\delta
})\rangle,\nabla(\rho^{-\delta})\bigr\rangle} {-\lvert\nabla u \rvert
^{2}\Delta u}= \frac{2\delta^{3}(\delta+1)\rho^{-3\delta-4}}{-\lvert\nabla u
\rvert^{2}\Delta u}>0,
\]
we can estimate the first term as
\begin{equation}
\label{T1}%
\begin{split}
&  \frac{R_{4}^{3\delta}{\mathcal{B}}(\lvert\nabla\varphi\rvert^{2}%
)\lvert\nabla\varphi\rvert^{2} \bigl\langle\nabla\langle\nabla(\rho^{-\delta
}),\nabla(\rho^{-\delta})\rangle,\nabla(\rho^{-\delta})\bigr\rangle}
{-\lvert\nabla u \rvert^{2}\Delta u}\\
&  \qquad\le\frac{\max(0,B_{0})R_{4}^{3\delta}\, T\lvert\nabla(\rho^{-\delta})
\rvert}{\lvert\nabla u \rvert}\\
&  \qquad\le\frac{2\max(0,B_{0})\lambda}{(1-\delta/R_{4}^{\delta}%
)^{3}(1-\lambda)}\\
&  \qquad\le\frac{2\max(0,B_{0})\lambda}{(1-\delta)^{3}(1-\lambda)}.
\end{split}
\end{equation}
The second term can be estimated by Lemma~\ref{perusta}, \eqref{(3.33)}, and
\eqref{est_T_denom} as
\begin{equation}
\label{T2}%
\begin{split}
&  \frac{\bar{B_{0}}\bigl(\lvert\nabla\langle\nabla h,\nabla h\rangle
\rvert+2R_{4}^{\delta}\lvert\nabla\langle\nabla h,\nabla(\rho^{-\delta})
\rvert\bigr)}{-\lvert\nabla u \rvert\Delta u}\\
&  \qquad\le\frac{\bar{B_{0}}\bigl(\lvert\nabla\langle\nabla h,\nabla
h\rangle\rvert+2R_{4}^{\delta}\lvert\nabla\langle\nabla h,\nabla(\rho
^{-\delta}) \rvert\bigr)(f_{a}\circ\rho)} {(R_{4}^{\delta}-\delta)^{2}%
\delta^{2}\rho^{-2\delta-2}(1-\lambda)(n-1)(f_{a}^{\prime}\circ\rho)}\\
&  \qquad\le\frac{\bar{B_{0}}(1+2R_{4}^{\delta})c_{6}\rho^{-C_{4}+2\delta}%
}{(R_{4}^{\delta}-\delta)^{2}\delta^{2}(1-\lambda)(n-1)}\\
&  \qquad\le\frac{3\bar{B_{0}}c_{6}\rho^{-C_{4}+2\delta}}{R_{4}^{\delta
}(1-\delta/R_{4}^{\delta})^{2}\delta^{2}(1-\lambda)(n-1)} \le\delta.
\end{split}
\end{equation}
The third term can be estimated by using \eqref{(3.30)} and \eqref{T_est} as
\begin{equation}
\label{T3}%
\begin{split}
&  \frac{\bar{B_{0}}R_{4}^{2\delta}\lvert\nabla\langle\nabla(\rho^{-\delta}),
\nabla(\rho^{-\delta})\rangle\rvert\lvert\nabla h \rvert} {-\lvert\nabla u
\rvert^{2}\Delta u} =\frac{\bar{B_{0}}R_{4}^{2\delta} T \lvert\nabla h \rvert
}{\lvert\nabla u \rvert}\\
&  \le\frac{2\bar{B_{0}}R_{4}^{2\delta}\delta\lambda}{(R_{4}^{\delta}%
-\delta)^{3}(1-\lambda)}\\
&  \le\frac{2\bar{B_{0}}\delta\lambda}{(1-\delta)^{3}(1-\lambda)}.
\end{split}
\end{equation}

Putting the estimates \eqref{T1}-\eqref{T3} and \eqref{delta_bound} together
we finally obtain
\[%
\begin{split}
&  \frac{{\mathcal{B}}(\lvert\nabla\varphi\rvert^{2})\lvert\nabla\varphi
\rvert^{2}\bigl\langle\nabla\langle\nabla\varphi,\nabla\varphi\rangle,
\nabla\varphi\bigr\rangle}{-\lvert\nabla\varphi\rvert^{2}\Delta\varphi}\\
&  \le\frac{2\max(0,B_{0})\lambda}{(1-\delta)^{3}(1-\lambda)} + \delta+
\frac{2\bar{B_{0}}\delta\lambda}{(1-\delta)^{3}(1-\lambda)}\\
&  \le\delta+ \frac{2\lambda\bigl(\max(0,B_{0})+\bar{B_{0}}\delta
\bigr)}{(1-\lambda)(1-\delta)^{3}} <1
\end{split}
\]
in $3\Omega\setminus B(o,R_{4})$. Hence $\varphi=A(R_{r}^{\delta}\rho
^{-\delta}+ h)$ is a continuous ${\mathcal{Q}}$-supersolution in
$3\Omega\setminus\bar B(o,R_{4})$.
\end{proof}


\section{Proof of Theorem \ref{ThmMain}}

\label{Proof_main}
Let $\tilde{f}\in C(\bar{M})$ be an extension of the given boundary data $f\in
C(M(\infty))$. Choose an exhaustion of $M$ by an increasing sequence of
${\mathcal{Q}}$-regular domains $\Omega_{k}$ provided by the assumption (A).
Hence there exist ${\mathcal{Q}}$-solutions $u_{k}\in C(\bar{\Omega}_{k})\cap
W^{1,p}_{\mathrm{loc}}(\Omega_{k})$ such that
\[%
\begin{cases}
{\mathcal{Q}}[u_{k}]=0 & \mbox{in }\Omega_{k},\\
u_{k}\vert\partial\Omega_{k}=\tilde{f}. &
\end{cases}
\]
Then
\[
-\max\lvert\tilde{f} \rvert\le u_{k}\le\max\lvert\tilde{f} \rvert
\]
in $\Omega_{k}$ by the Comparison principle (Lemma~\ref{lem_comp}). Condition
(B) together with a diagonal argument implies that there exists a subsequence,
still denoted by $u_{k}$, that converges locally uniformly in $M$ to a
${\mathcal{Q}}$-solution $u\in C(M)$. Therefore the proof of
Theorem~\ref{ThmMain} reduces to prove that $u$ extends continuously to
$M(\infty)$, satisfies $u\vert M(\infty)=f$, and is the unique ${\mathcal{Q}}%
$-solution with boundary values $f$. To this end, let $x_{0}\in M(\infty)$ and
$\varepsilon>0$. Since $f$ is continuous, there exists $L\in(8/\pi,\infty)$
such that
\[
\lvert f(y)-f(x_{0}) \rvert<\varepsilon/2
\]
for all $y\in C(v_{0},4/L)\cap M(\infty)$, where $v_{0}=\dot{\gamma}%
_{0}^{o,x_{0}}$ is the initial vector of the geodesic ray representing $x_{0}%
$. We claim that
\begin{equation}
\label{w<u<v}w(x):=-\varphi(x) +f(x_{0})-\varepsilon\le u(x)\le v(x):=\varphi
(x)+f(x_{0})+\varepsilon
\end{equation}
in $U=3\Omega\setminus\bar{B}(o,R_{4})$, where $\varphi=A(R_{4}^{\delta}%
\rho^{-\delta}+h)$ is the ${\mathcal{Q}}$-supersolution in $U$ as in
Lemma~\ref{supsol}, with $A=2\max_{\bar{M}}\lvert\tilde{f} \rvert$. Note that
$-\varphi$ is a ${\mathcal{Q}}$-subsolution in $U$. Recall the notation
$\Omega=C(v_{0},1/L)\cap M$ and $k\Omega=C(v_{0},k/L)\cap M,\ k>0,$ from
Section~\ref{sec_preli}. Since $\tilde{f}$ is continuous in $\bar{M}$, there
exists $k_{0}$ such that
\begin{equation}
\label{part_omega_k}\lvert\tilde{f}(x)-f(x_{0}) \rvert<\varepsilon/2
\end{equation}
for all $x\in\partial\Omega_{k}\cap U$ and all $k\ge k_{0}$ and that
$\partial\Omega_{k_{0}}\cap U\ne\emptyset$. Let $V_{k}=\Omega_{k}\cap U$ for
$k\ge k_{0}$. We have
\[
\partial V_{k}=( \partial\Omega_{k}\cap\bar{U} )\cup(\partial U\cap\bar
{\Omega}_{k}).
\]
Next we will show by using the Comparison principle that
\begin{equation}
\label{w<u_k<v}w\le u_{k}\le v
\end{equation}
in $V_{k}$. By \eqref{part_omega_k}, we have
\[
w(x)\le f(x_{0})-\varepsilon/2\le\tilde{f}(x)=u_{k}(x)\le f(x_{0}%
)+\varepsilon/2\le v(x)
\]
for all $x\in\partial\Omega_{k}\cap\bar{U}$ and $k\ge k_{0}$. On the other
hand,
\[
h\vert M\setminus(2\Omega\cup B(o,R_{1})=1
\]
by Lemma~\ref{arvio_lause} and $R_{4}^{\delta}\rho^{-\delta}=1$ on $\partial
B(o,R_{4})$, and therefore $\varphi\ge A=2\max_{\bar{M}}\lvert\tilde{f}
\rvert$ on $\partial U\cap\bar{\Omega}_{k}$. It follows that
\[
v=\varphi+f(x_{0})+\varepsilon\ge2\max_{\bar{M}}\lvert\tilde{f} \rvert
+f(x_{0})+\varepsilon\ge\max_{\bar{M}}\lvert\tilde{f} \rvert+\varepsilon\ge
u_{k}
\]
on $\partial U\cap\bar{\Omega}_{k}$. Similarly, $u_{k}\ge w$ on $\partial
U\cap\bar{\Omega}_{k}$. Thus $w\le u_{k}\le v$ on $\partial V_{k}$ and
\eqref{w<u_k<v} follows. Since this holds for all $k\ge k_{0}$, we obtain
\eqref{w<u<v}. Finally,
\[
\limsup_{x\to x_{0}}\lvert u(x)-f(x_{0}) \rvert\le\varepsilon
\]
since $\lim_{x\to x_{0}}\varphi(x)=0$. Thus $u$ extends continuously to
$C(\bar{M})$ and $u\vert M(\infty)=f$ since $x_{0}\in M(\infty)$ and
$\varepsilon>0$ were arbitrary. We are left with the uniqueness of $u$.
Therefore, let $\tilde{u}\in C(\bar{M})$ be another ${\mathcal{Q}}$-solution
in $M$, with $\tilde{u}=u=f$ in $M(\infty)$. Suppose on the contrary that
$\tilde{u}\ne u$. Thus we may assume without loss of generality that
$\tilde{u}(x)>u(x)+\varepsilon$ for some $x\in M$ and $\varepsilon>0$. Let $D$
be the $x$-component of the set $\{y\in M\colon\tilde{u}(y)>u(y)+\varepsilon
\}$. Then $D$ is open with compact closure since both $\tilde{u}$ and $u$ are
continuous in $\bar{M}$ and coincide on $M(\infty)$. Furthermore, $\tilde
{u}=u+\varepsilon$ on $\partial D$, and therefore $\tilde{u}=u+\varepsilon$ in
$D$ by Corollary~\ref{cor_uniq} which leads to a contradiction with $\tilde
{u}(x)>u(x)+\varepsilon$. This concludes the proof of Theorem~\ref{ThmMain}.

\bibliographystyle{acm}


\end{document}